\documentclass[12pt]{article}
\usepackage{}

\usepackage{ulem}
\usepackage{epsfig}
\usepackage{latexsym}
\usepackage{caption}
\usepackage{amsmath}
\usepackage{amsfonts}
\usepackage{amssymb}
\usepackage{graphicx}
\usepackage{mathrsfs}
\usepackage{enumerate}
\usepackage{graphics}
\usepackage{MnSymbol}
\usepackage{float}
\usepackage{pict2e}
\usepackage{tikz}
\usepackage{color}

\makeatletter

\newcommand{\Rmnum}[1]{\expandafter\@slowromancap\romannumeral #1@}

\makeatother

\begin{document}

\newtheorem{theorem}{Theorem}
\newtheorem{observation}[theorem]{Observation}
\newtheorem{corollary}{Corollary}
\newtheorem{algorithm}[theorem]{Algorithm}
\newtheorem{definition}{Definition}
\newtheorem{guess}[theorem]{Conjecture}
\newtheorem{claim}{Claim}
\newtheorem{problem}[theorem]{Problem}
\newtheorem{question}[theorem]{Question}
\newtheorem{lemma}{Lemma}
\newtheorem{proposition}[theorem]{Proposition}
\newtheorem{fact}[theorem]{Fact}

\makeatletter
  \newcommand\figcaption{\def\@captype{figure}\caption}
  \newcommand\tabcaption{\def\@captype{table}\caption}
\makeatother

\newtheorem{acknowledgement}[theorem]{Acknowledgement}

\newtheorem{axiom}[theorem]{Axiom}
\newtheorem{case}[theorem]{Case}
\newtheorem{conclusion}[theorem]{Conclusion}
\newtheorem{condition}[theorem]{Condition}
\newtheorem{conjecture}[theorem]{Conjecture}
\newtheorem{criterion}[theorem]{Criterion}
\newtheorem{example}[theorem]{Example}
\newtheorem{exercise}[theorem]{Exercise}
\newtheorem{notation}[theorem]{Notation}
\newtheorem{solution}[theorem]{Solution}
\newtheorem{summary}[theorem]{Summary}

\newenvironment{proof}{\noindent {\bf
Proof.}}{\rule{3mm}{3mm}\par\medskip}
\newcommand{\remark}{\medskip\par\noindent {\bf Remark.~~}}
\newcommand{\pp}{{\it p.}}
\newcommand{\de}{\em}
\newcommand{\mad}{\rm mad}
\newcommand{\qf}{Q({\cal F},s)}
\newcommand{\qff}{Q({\cal F}',s)}
\newcommand{\qfff}{Q({\cal F}'',s)}
\newcommand{\f}{{\cal F}}
\newcommand{\ff}{{\cal F}'}
\newcommand{\fff}{{\cal F}''}
\newcommand{\fs}{{\cal F},s}
\newcommand{\s}{\mathcal{S}}
\newcommand{\G}{\Gamma}
\newcommand{\g}{\gamma}
\newcommand{\wrt}{with respect to }

\newcommand{\qq}{\uppercase\expandafter{\romannumeral1}}
\newcommand{\qqq}{\uppercase\expandafter{\romannumeral2}}

\newcommand{\qed}{\hfill\rule{0.5em}{0.809em}}

\title{List colouring triangle free planar graphs}

\author{ Jianzhang Hu \and Xuding Zhu\thanks{Department of Mathematics, Zhejiang Normal University,  China.  E-mail: xudingzhu@gmail.com. Grant Numbers: NSFC 11571319  and 111 project of Ministry of Education of China.}}

\maketitle

\begin{abstract}
	This paper proves   the following result: Assume $G$ is a triangle free planar graph, $X$ is an independent set of $G$. If $L$ is a list assignment of $G$ such that $\mid L(v)\mid = 4$  for each vertex $v \in V(G)-X$ and  $\mid L(v)\mid = 3$  for each vertex $v \in X$, then $G$ is $L$-colourable.
\end{abstract}
\section{Introduction}
\label{sec-intro}
Given a graph $G$ and a function $f$ from $V(G)$ to $N$. An {\em   $f$-list assignment} of $G$ is a mapping $L$ which assigns to each vertex $v$ of $G$ a set $L(v)$ of $f(v)$ integers as available colours. Given a list assignment $L$ of $G$, an {\em $L$-colouring} of $G$ is a mapping $\phi: V(G)\rightarrow N $ such that $\phi(v)\in L(v)$ for each vertex $v$ and $\phi(x)\neq \phi(y)$ for each  edge $xy$. We say $G$ is $L$-colourable if there exists an $L$-colouring of $G$. A {\em $k$-list assignment} of $G$ is an $f$-list assignment of $G$ with $f(v)=k$ for all $v \in V(G)$. We say $G$ is  {\em $k$-choosable} if $G$ is $L$-colourable for all $k$-list assignment $L$ of $G$.

List colouring of graphs was introduced in the 1970's by Vizing \cite{Viz1976}  and independently by Erd\H{o}s, Rubin and Taylor \cite{ERT1979}. Thomassen  showed that every planar graph is 5-choosable \cite{Tho1994}, and   every planar graph with girth at least $5$ is $3$-choosable \cite{Tho2000}. A classical result of
Gr\"{o}tzsch \cite{Gro1958} says that every triangle free planar graph   is 3-colourable.  However, Voigt \cite{Voi1995} constructed  a triangle free planar graph   which is not 3-choosable.
Kratochv\'{i}l and Tuza
\cite{KT1994} observed that every  triangle free planar graph is 4-choosable. A natural question is how far can a triangle free planar graph from being $3$-choosable? To make this question more precise, we propose the following conjecture.

\begin{conjecture}
	\label{conj1}
	Assume $G$ is a triangle free planar graph, $X$ is a subset of $V(G)$ such that $G[X]$ is a bipartite graph, and $L$ is a list assignment of $G$ which assigns to each vertex in   $X$ a set of $3$ permissible colours and assigns to each other vertex  a set of  $4$ permissible colours. Then $G$ is $L$-colourable.
\end{conjecture}

It is known that the conjecture is true if $G$ itself is bipartite, i.e., bipartite triangle free planar graphs are $3$-choosable \cite{Alon1992}.
In this paper, as a support of Conjecture \ref{conj1}, we prove the following result:

\begin{theorem}
	\label{theorem1}
	Assume $G$ is a triangle free planar graph and $X$ is an independent set of $G$. If $L$ is a list assignment of $G$ such that  $|L(v)|=3$ for $v \in X$ and $|L(v)| = 4$  for   $v \in V(G)-X$, then   $G$ is $L$-colourable.
\end{theorem}

For a plane graph $G$, denote by $B(G)$ the boundary of the infinite face of $G$. The vertices of $B(G)$ are cyclically ordered. For $v \in B(G)$, denote by $v^-,v^+$ the  neighbors of $v$ that are preceding and succeeding  $v$ along  the clockwise cyclic order, respectively.

Theorem \ref{theorem1} is a consequence of the following more technical result.

\begin{definition}
	\label{def-target}
	A {\em target} is a  triple $(G,P,L)$ such that $G$ is a triangle free plane graph,  $P=p_1p_2\ldots p_k$ is a path consisting of   $k \le 5$ consecutive vertices of $B(G)$ and $L$ is a list assignment of $G$ such that the   following hold:
	\begin{itemize}
		\item  $ 3\le  |L(v)| \le 4$  for    $v\in V(G)- B(G)$.
		\item  The set $X_{G,L} = \{x \in V(G)-B(G): |L(x)| =3\} $ is an independent set.
		\item $2 \le |L(v)| \le 4 $  for   $v\in B(G) - V(P)$.
		\item $ |L(v)| =1$  for   $v\in  V(P)$ and if $u, v \in P$, and $uv \in E(G)-E(P)$, then $L(u) \ne L(v)$. (For $uv \in E(P)$, we do not require $L(u) \ne L(v)$).
	\end{itemize}
\end{definition}
\begin{definition}
	\label{def-valid}
	Assume $(G,P,L)$ is a target.
	\begin{itemize}
		\item A vertex $u$ of $G$ is   bad  if either $|P|=5$   and
$|N_G(u)\cap P| \ge |L(u)|-1$, or 
 $3\le |P| \le 4$, $|L(u)|=2$ and  $u$ is adjacent to an end vertex of $P$.
\item An edge $e=xy$ is bad if $|L(x)| = |L(y)|=2$.
\item A  4-cycle $C=xyzw$ is bad if (1) $w$ is an interior vertex and $x,y,z$ are    consecutive boundary vertices; (2) $|L(x)|=|L(w)|=|L(z)|=3$ and $|L(y)|=2$.
\end{itemize}
A target $(G,P,L)$ is {\em valid} if there is no bad edge, no bad vertex and no bad $4$-cycle.
\end{definition}


\begin{theorem}
	\label{theorem2}
	If $(G,P,L)$ is a valid target, then
	$G-E(P)$ is $L$-colourable.
\end{theorem}

\section{Some properties of a minimum counterexample}
\label{sec-proof}
Assume Theorem \ref{theorem2} is not true and $(G,P,L)$ is a counterexample such that
\begin{itemize}
	\item[(1)] $|V(G)|$ is minimum.
	\item[(2)]  Subject to (1), $\sum_{v \in V(G)}|L(v)|$ is minimum.
\end{itemize}

We say a vertex $v$ is {\em pre-coloured}  if $|L(v)|=1$.   For technical reasons, we do not require the vertices of $P$ be properly coloured, so the final colouring of $V(G)$ is a proper colouring of $G -E(P)$. However, when we pre-colour a vertex $v$ in the proof, we always colour $v$  by a colour not used by any of its coloured neighbours.    We say $L'$ is obtained from $L$ by {\em pre-colouring $v$ } to mean that
$L'(v)=\{c\}$ for some   colour $c \in L(v)$ and no pre-coloured neighbour of $v$ is coloured by $c$.

  \begin{lemma}
  	\label{lemmaldv}
  	For any vertex $v$ of $G$, $|L(v)| \le d_G(v)$ .
  \end{lemma}
  Otherwise, by the minimality of $G$, $G-v-E(P)$ has a proper $L$-colouring, which can be extended to a proper $L$-colouring of $G-E(P)$ by colouring $v$ with a colour from $L(v)$ not used by its neighbours.

\begin{lemma}
	\label{lemma2conn}
	$G$ is 2-connected.
\end{lemma}
\begin{proof}
	Assume to the contrary that $G$ has a cut vertex $v$. Let $G_{1}$ and $G_{2}$ be the two induced subgraphs  of $G$ with $V(G_1) \cap V(G_{  2}) = \{v\}$ and $V(G_{ 1}) \cup V(G_{ 2})=V(G)$. If $v\in V(P)$, then by the minimality of $G$ there exist   $L$-colourings of $G_{1}-E(P)$ and $G_{2}-E(P)$, whose union is an $L$-colouring of $G-E(P)$,  a contradiction.
	
	Assume $P\subseteq G_{1}$.  By the minimality of $G$ there is an $L$-colouring $\phi$ of $G_{1}$.

	Let $L_2$ be the restriction of $L$ to $G_{ 2}$, except that $v$ is pre-coloured, and let $P'=v$. Then  $(G_{ 2}, P', L_2)$ has no bad vertex, no bad edge and no bad $4$-cycle, and $X_{G_2,L_2} \subseteq X_{G,L}$. So $(G_{ 2}, P', L_2)$ is a valid target and $G_{ 2}$ has an $L_2$-colouring
	$\psi$. The union of $\phi$ and $\psi$ is an $L$-colouring of $G-E(P)$.
\end{proof}

In the remainder of the proof, we usually need to verify that
some triple $(G',P',L')$ is a valid target, where $G'$ is a subgraph of $G$ and $L'$ is a slight modification of the restriction of $L$ to $G'$ and $P'$ is the set of pre-coloured vertices in $G'$. It is usually obvious  that $X_{G',L'} \subseteq X_{G,L}$, and hence $X_{G',L'}$ is an independent set in $G'$. The other conditions for $(G',P',L')$ to be a target are also trivially satisfied.  We shall only need to prove that $(G',P',L')$  has no bad $4$-cycle, no bad edge and no bad vertex.

For a cycle $C$ of $G$, we denote by ${\rm int}[C]$ the subgraph of $G$ induced by vertices on $C$ and  in the interior of $C$. The subgraph ${\rm ext}[C]$ is defined similarly. We denote by ${\rm int}(C)$ the subgraph of $G$ induced by vertices in the interior of $C$.

\begin{lemma}
	\label{lemmasep45}
	$G$ has no separating 4-cycle and  5-cycle.
\end{lemma}
\begin{proof}
	Let $C$ be a separating 4-cycle or 5-cycle of $G$. Let $G_1={\rm ext}[C]$ and $G_2={\rm int}[C]$. We choose the cycle $C$ so that $G_2$ has minimum number of vertices.
	By the minimality of $G$,  $G_1-E(P)$ has an $L$-colouring $\phi$. Let $L_2$ be the restriction of $L$ to $G_2$, except that
	vertices on $C$ are pre-coloured by $\phi$.

	As no vertex $v$ of $G_2$ has $|L_2(v)| =2$, $G_2$ has no bad edge, and no bad $4$-cycle. If $G_2$ has no bad vertex, then  $(G_2, C, L_2)$ is a valid target, and
	by the minimality of $G$, $G_2 - E(C)$ has an $L_2$-colouring $\psi$. Note that the cycle $C$ is properly coloured by $\phi$.
	Therefore the   union of $\phi$ and $\psi$ is an $L$-colouring of $G-E(P)$,  a contradiction.
	
	Assume $u$ is a bad vertex of $G_2$. Then $u$ has two neighbors in $C$ (if $u$ has three neighbors in $C$, then $G$ would contain a triangle).  By the minimality of $G_2$, $u$ is the only vertex of ${\rm int}(C)$. Hence  $d_G(u) =2 < |L(u)|$, contrary to Lemma \ref{lemmaldv}.
\end{proof}

 \begin{lemma}
  	 	\label{lemmasep6}
  	 	If $C$ is a separating 6-cycle of $G$,  then ${\rm int}(C)=\{z\} $ and $z\in X_{G,L}$.
  	 \end{lemma}
  	 \begin{proof}
  	 	Assume $C=x_1x_2x_3x_4x_5x_6$ is a separating cycle of $G$ for which $|{\rm int} (C)| \ge 2$ and subject to this condition, ${\rm int}(C)$ is minimal.
  	 	
  	 	By the minimality of $G$, there is an   $L$-colouring $\phi$ of ${\rm ext}[C]-E(P)$.
  	 	
  	 	If there exist a vertex $w \in {\rm int}(C)$ adjacent to at least two vertices in $V(C)$, then since $G$ has no separating $4$-cycle or $5$-cycle, and $|{\rm int}(C)| \ge 2$, we conclude that $w$ is adjacent to two vertices of $C$ that are of distance $2$ on $C$, say $w$ is adjacent to $x_1$ and $x_3$. For $C_1=x_1x_2x_3w$ and $C_2=x_3x_4x_5x_6x_1w$, ${\rm int}(C_1) =\emptyset$ and ${\rm int}(C_2)$ contains a unique vertex $z$ (by the minimality of ${\rm int}(C)$).   By Lemma \ref{lemmaldv}, $d_G(w) \ge |L(w)| \ge 3$. As   $w$ is adjacent to at most two vertices of $C$, we conclude that  $zw$ is an edge of $G$ and $d_G(w)=|L(w)|=3$.  By Lemma \ref{lemmaldv},  we have $d_G(z)=|L(z)|=3$. This contradicts the assumption that $X_{G,L}$ is an independent set.

  	 	Assume  no  vertex of  ${\rm int}(C)$ is adjacent to two vertices in $V(C)$.

  	 	Let $G'= {\rm int}[C]- \{x_6\}$ and $P'=x_1x_2x_3x_4x_5$,
  	 	and let $L'$ be the restriction of $L$ to $G'$, except that vertices of $P'$ are pre-coloured and
  	 	for any neighbour $u$ of $x_6$, $L'(u)=L(u) -\{\phi(x_6)\}$.
  	 	
  	 	As any vertex  $w \in {\rm int}(C)$ is adjacent to at most one vertex in $V(C)$, we know that $(G', P', L')$ has no bad vertex and no bad edge and no bad $4$-cycle.  Hence  $(G', P', L')$ is  a valid target, and  $G'-E(P')$ has an $L'$-colouring $\psi$. The union of $\phi$ and $\psi$ form an $L$-colouring of $G-E(P)$.  Hence ${\rm int}(C)$ contains a unique vertex $z$.  As $ 3 \le |L(z)| \le 4$ and  $z$ has at most three neighbours on $C$ and by Lemma \ref{lemmaldv},  $|L(z)|=3$ and  $z \in X_{G,L}$.
  	 \end{proof}

For $i=1,2,3,4$, an {\em  $i$-chord}   of $G$ is a path $W$ of length $i$ (i.e., has $i$ edges) such that the two end vertices  are two non-consecutive boundary vertices and the other vertices of  $W$ are interior vertices of $G$. A $1$-chord of $G$ is also called a chord of $G$.
An $i$-chord $W$ with end vertices $u$ and $v$ is called an $(a,b)$-$i$-chord
(respectively, an $(a, b^+)$-$i$-chord) if $|L(u)|=a$ and $|L(v)|=b$ (respectilvey, if $|L(u)|=a$ and $|L(v)| \ge b$. 

	Assume $W$ is an $i$-chord of $G$ with end vertices $u$ and $v$.   The two components of $W$ are induced subgraphs $G_{W,1}, G_{W,2}$  of $G$ such that $V(G_{W,1})\cap V(G_{W,2})= V(W)$   and $V(G_{W,1}) \cup V(G_{W,2}) = V(G)$. For convenience, we shall always assume that
	
\begin{itemize}
	\item[(1)]  $P \cap G_{W,1}$ is a path on the boundary of $G_{W,1}$.
	\item[(2)] Subject to (1),  $|V(G_{W,1}) \cap P|$ is maximum.
\end{itemize}
	We name the vertices so that   $u^-, v^+$  are vertices in  $G_{W,2}$.

	 As  $P \cap G_{W,1}$ is a path on the boundary of $G_{W,1}$,  $(G_{W,1}, P \cap G_{W,1}, L)$ is a valid target.
	 By the minimality of $G$, there is an   $L$-colouring $\phi$ of $G_{W,1}- E(P)$.

\begin{definition}
	\label{def-PW}
	Assume $W$ is an $i$-chord of $G$ with end vertices $u$ and $v$. Let $P_W$ be the subpath of $B(G_2)$ which consists of  $(W \cup P) \cap B(G_2)$, and moreover,  if $|L(u^-)|=2$ (respectively,  $|L(v^+)|=2$), then  $u^- \in P_W$ (respectively,   $v^+ \in P_W$).
	Let $G'_{W,1}$ be the subgraph of $G$ induced by $(V(G_{W,1}) \cup P_W)$. We say $W$ is {\em feasible} if $G'_{W,1} \ne G$. 
		If $W$ is a  feasible  and $\phi$ is an   $L$-colouring of $G'_{W,1}-E(P)$, then we denote by $L_{\phi}$ the restriction of $L$ to $G_{W,2}$, except that for $x \in P_W$, $L'(x) = \{\phi(x)\}$, i.e., vertices of $P_W$ are pre-coloured.  
\end{definition}

Note that if
	$W$ is a feasible $i$-chord and $\phi$ is an   $L$-colouring of $G'_{W,1}-E(P)$, then  $(G_{W,2}, P_W, L_{\phi})$ is not a valid target, for otherwise $G_{W,2}-E(P_W)$ has an $L_{\phi}$-colouring $\psi$, and the union of $\phi$ and $\psi$ is an $L$-colouring of $G-E(P)$.

{For convenience, when $W$ is a feasible $i$-chord, we may write $(G_{W,2}, P_W, L)$ for $(G_{W,2}, P_W, L_{\phi})$, where $\phi$ is not explicitly specified and vertices of $P_W$ are pre-coloured.}

\begin{lemma}
   \label{lemmapw5}
Assume $W$ is a feasible $i$-chord   of $G$, then $|P_W| \ge 5$.
\end{lemma}
\begin{proof}
	Assume to the contrary that there is a  feasible $i$-chord $W$
	with $P_W = p'_1p'_2 \ldots p'_t$ for some  $t \le 4$. We
	choose such an $i$-chord $W$ so that $G_{W,2}$ has minimum number of vertices.
	Let $\phi$ be an $L$-colouring  of $G'_{W,1}-E(P)$.  Then $(G_{W,2}, P_W, L_{\phi})$ is a target.

	  As $|L_{\phi}(x)| \ge 2$ implies that $L_{\phi}(x) = L(x)$, we conclude that   $(G_{W,2}, P_W, L_{\phi})$ has no bad $4$-cycle and no bad edge. By definition of $P_W$, there is no vertex $u$ adjacent to an end vertex of $P_W$ and with $|L(u)|=2$, i.e.,  $(G_{W,2}, P_W, L_{\phi})$ has no bad vertex, a contradiction.
\end{proof}

  	\begin{lemma}
     \label{lemmanoichord}
 There exists no $i$-chord as described below:
  \begin{itemize}
  	\item  a $(1,1)$-$1$-chord,  a $(2, 1^+)$-$1$-chord, a $(2,2^+)$-$2$-chord, a  $(2,2)$-$3$-chord,  {a $(3,3)$-1-chord.}
  	 \item  a  $(1,1^+)$-$1$-chord $W$ with $|P \cap B(G_{W,2})| \le 2$, a $(1,1^+)$-$2$-chord with   $|P \cap B(G_{W,2})| = 1$,   a $(1,2)$-$2$-chord $W$ with $|P \cap B(G_{W,2})| \le  2$. 
  \end{itemize}
  	\end{lemma}	
  	\begin{proof}
  		  If $W$ is an $i$-chord listed above,   then it is straightforward to verify   that $W$ is feasible and $|P_W| \le 4$, contrary to Lemma \ref{lemmapw5}. 
  	\end{proof}

  	\begin{lemma}
  		\label{lemmabadv}
  		Assume $W$ is a feasible $i$-chord for some $i=1,2,3,4$, $\phi$ is an $L$-colouring of $G_{W,1}$. If $|P_W|=5$, then $(G_{W,2}, P_W, L_{\phi})$  has a bad vertex $x$ with $|L(x)|=3$, and $|N_G(x) \cap P_W|\ge 2$. Moreover, if $|P_W \cap P|=3$, then   $y$ has exactly one neighbour in $P_W \cap P$ and one neighbour in $P_W-P$.	
  	\end{lemma}
  	\begin{proof}
  		Let $\phi$ be an $L$-colouring of $G'_{W,1}-E(P)$. It follows from the definition that for any vertex $x$ of $G_{W,2}$, if $|L_{\phi}(x)| \ge 2$, then
  		$L_{\phi}(x)=L(x)$. Hence $(G_{W,2}, P_W, L_{\phi})$ has no bad 4-cycle and no bad edge. Assume $P_W=p'_1p'_2\ldots p'_5$ and  $x$ is a bad vertex, then we may assume that $p'_jx$ is an edge where $j \le 3$. If $|L(x)|=2$, then $x$ is a boundary vertex of $G$ and it is easy to verify that $p'_1 \dots p'_jx$ contains a  feasible $i$-chord $W'$ for some $i \le j$ and with $|P_{W'}| \le 4$, contrary to Lemma \ref{lemmapw5}.  If $|L(x)|=4$, then $x$ is adjacent to $p'_1,p'_3, p'_5$, and either $x$ is a boundary vertex and $xp'_1$ or $xp'_5$ is a chord, or $x$ is an interior vertex and $W'=p'_1xp'_5$ is a 2-chord with $|P_{W'}|=3$, contrary to Lemma \ref{lemmapw5}.      Thus $|L(x)|=3$ and  $|N_G(x) \cap P_W|\ge 2$.   For the moreover part, assume $P_W   =  p_1   p_2  p_3  p'_4  p'_5 $. If $y$    has two neighbours in $P_W \cap P$, then $y$ is adjacent to $p_1,p_3$ and   $|L(y) - (L(p_1) \cup L(p_3))|=1$. Thus $y$ is a bad vertex in $(G,P,L)$, a contradiction.   Hence $y$ has exactly one neighbour in $P_W \cap P$ and one neighbour in $P_W-P$.
  	\end{proof}

  	 \begin{lemma}
  	 	\label{lemmano332chord}
  	{	There is no feasible $(3,3)$-2-chord $uvw$ such that $|L(v)|=3$.}
  	 \end{lemma}
  	 \begin{proof}
 {	 
  	Assume $W=uvw$ is a feasible $(3,3)$-2-chord with $|L(v)|=3$. By Lemma \ref{lemmapw5}, $|P_W|=5$ and hence
  	 	 $P_W=u^+uvww^-$ and $|L(w^-)|=|L(v')|=2$. By Lemma \ref{lemmabadv},   $(G_{W,2}, P_W, L)$ has a bad vertex $x$ with $|L(x)|=3$ and $|N_G(x) \cap P_W|\ge 2$. If $x$ is   adjacent to $v$, then $x$ is a boundary vertex, and is adjacent to one of $u^+,w^-$ and hence $G$ has a bad 4-cycle, a contradiction. If $x$ is adjacent to $u^+$ and $w^-$, then either $x$ is a boundary vertex and $x$ is contained in a separating $6$-cycle with a unique vertex $y$ adjacent to $x,w,u$, and hence $(G,P,L)$ has a bad 4-cycle, or $x$ is an interior vertex and $u^+xw^-$ is a $(2,2)$-chord, contrary to Lemma \ref{lemmanoichord}. If $x$ is adjacent to $u^+$ and $w$, then 
      $u^+xw$  is a $(2,3)$-2-chord or a $(2^+,2)$-chord, contrary to Lemma \ref{lemmanoichord}. Similarly, $x$ is not adjacent to $w^-$ and $u$. Hence $x$ is adjacent to $u$ and $w$. Then $W'=uxw$ is a feasible $(3,3)$-2-chord, and $|L(x)|=3$. Now $(G_{W',2},P_{W'},L)$ has a bad vertex $x'$. If $x'$ is not adjacent to $x$, then $x'$ would be contained in a separating cycle $C$, with $x$ be the unique interior vertex in ${\rm int}(C)$. But then $d_G(x)=2$, a contradiction. If $x'$ is adjacent to $x$, then $x'$ is a boundary vertex adjacent to one of $u^+$ and $w^-$, and hence $G$ has a bad 4-cycle in $(G,P,L)$,  a contradiction. }
  	\end{proof}
  	
  	\begin{lemma}
  		\label{lemma-34chord}
  		{	There is no feasible $(3,4)$-2-chord $uvw$ such that $|L(v)|=3$.}
  	\end{lemma}
  	\begin{proof}
  		{	 
  	 Assume that $W=uvw$ is a feasible $(3,4)$-2-chord with $|L(v)|=3$. We choose $W$ so that $G_{W,2}$ is minimum.
  	 Then $P_W=u^+uvww^-$ where $|L(w^-)|=|L(u^+)|=2$ and  $(G_{W,2}, P_W, L)$ has a bad vertex $x$. If $x$ is adjacent to $u$ and $w$, then $W'=uxw$ is a feasible $(3,4)$-2-chord,  $|L(x)|=3$ and $G_{W',2}$ is smaller than $G_{W,2}$, contrary to the choice of $W$. }
  	 
  {	 If $x$ is adjacent to $v$, then $x$ is a boundary vertex and $W'=wvx$ is a $(3,3)$-2-chord with $|L(v)|=3$. By Lemma \ref{lemmano332chord}, $W'$ is not feasible,  hence $x$ is adjacent to $w^-$, but then $xvww^-$ is a bad 4-cycle, a contradiction.  } 
  	 
 {	   If $x$ is adjacent to both $u^+$ and $w^-$, then since $G$ has no $(2,2)$-2-chord, we conclude that $x$ is a boundary vertex, but then since $d_G(x) \ge 3$, the 6-cycle $u^+uvww^-x$ contains a unique interior vertex $z$ adjacent to $x, u$ and $w$. Then $xzww^-$ is a bad 4-cycle, a contradiction. }
  	 
  	{  Thus $x$ is adjacent to exactly one of $u^+,w^-$ and one of $w,u$. But then $wxu^+$ or $uxw^-$ is a 2-chord that contradicts Lemma \ref{lemmanoichord} or the choice of $W$. }
  	 \end{proof}

  	 	 \begin{lemma}
  	 	 	\label{lemma-13chord}
  	 	 	{There is no feasible   $(1,3)$-2-chord $uvw$ with $|P \cap B(G_{W,2})| \le 2$ and $|L(v)|=3$.}
  	 	 \end{lemma}
  	 	 \begin{proof}
  	 	 	{	 	Assume $W=uvp_2$ is a feasible $(1,3)$-2-chord with $|L(v)|=3$. By Lemma \ref{lemmanoichord},  $|P \cap B(G_{W,2})| =2$.  Assume $p_1,p_2 \in G_{W,2}$. We choose such a 2-chord so that $G_{W,2}$ is minimum. 
  	 	 		As $W$ is feasible, $u$ is not adjacent to $p_1$.   By Lemma \ref{lemmapw5} and Lemma \ref{lemmabadv}, $P_W=p_1p_2vuu^+$ where $|L(u^+)|=2$ and  $(G_{W,2}, P_W, L)$ has a bad vertex $x$. By the minimality of $G_{W,2}$, $x$ is not adjacent to $p_2$ and $u$. Also $x$ is not adjacent to $p_1$ and $u$, for otherwise either $x$ is an interior vertex and $W'=p_1xu$ is a $(1,3)$-2-chord with $|P \cap B(G_{W',2})|=1$, or $x$ is a boundary vertex and $xu$ is a $(3,3)$-chord, both in contrary to Lemma \ref{lemmanoichord}. 
  	 	 		Thus $x$ is   adjacent to $v$, but then $x$ is a boundary vertex and $xvu$ is a $(3,3)$-2-chord with $|L(v)|=3$, contrary to Lemma \ref{lemmano332chord}.    }
  	 	 \end{proof}
  	 
  \begin{lemma}
     \label{lemma-no(1,2)-2chord}
   There is no $(1,2)$-$2$-chord. 
  	\end{lemma}	
  	\begin{proof}
  			Assume $W=vxu$ is a  $(1,2)$-$2$-chord, with $v$ be a pre-coloured vertex.  By Lemma \ref{lemmanoichord}, $|P \cap B(G_{W,2})| = |P \cap B(G_{W,1})| =3$. Thus $P=p_1p_2p_3p_4p_5$ and   $W=p_3xu$. We may assume that 
  	$p_1,p_2,p_3$ is contained in $G_{W,2}$ and $p_3,p_4,p_5$ is contained in $G_{W,1}$.

  { Observe that either $G_{W,1}$ contains no vertex $y$ with $\{u, p_i\} \subseteq N_G(y)$ for some $i \in \{3,4\}$ or  $G_{W,2}$ contains no vertex $y$ with $\{u, p_i\} \subseteq N_G(y)  $ for some $i \in \{3,2\}$, for otherwise we would have $d_G(x)=2$, a contradiction. By symmetry, we may assume that  $G_{W,2}$ contains no vertex $y$ with $\{u, p_i\} \subseteq N_G(y)  $ for some $i \in \{3,2\}$.   }

 {It is easy to verify that $W$ is feasible.    By Lemma \ref{lemmabadv},  
  	$(G_{W,2}, P_W, L  )$ has a bad vertex  $y$. If $y$ is adjacent to $p_1$ and $u$, then either $y$ is a boundary vertex and $p_1y$ is a $(1,2^+)$-1-chord, or $uy$ is a $(2, 1^+)$-1-chord, or $y$ is an interior vertex and $p_1yu$ is a $(1,2)$-2-chord, contrary to Lemma \ref{lemmanoichord}.  	If $y$ is adjacent to $p_1$ and $x$, then let $W'=p_1yxv$, we can check that $W'$ is feasible and $|P_{W'}|=4$,  contrary to Lemma \ref{lemmapw5}. }
 
 { 	Thus $N_G(y) \cap P_W = \{x, p_2\}$. Let $W'=p_2yxu$.  We can check that $W'$ is a feasible $3$-chord, and  $P_{W'} = p_1p_2yxu$. 
  	By Lemma \ref{lemmabadv},  $(G_{W',2}, P_{W'}, L )$ has a bad vertex $z$.  If $z$ is not adjacent to $y$, then $z$ is contained in a separating cycle $C$ and $y$ is the unique vertex in ${\rm int}(C)$, and hence $d_G(y)=2$, a contradiction. Thus $z$ is adjacent to $y$, and hence $z$ is a boundary vertex, and $z$ is adjacent to at least one of $u$ and $p_1$.   }

 {If $z$ is adjacent to $u$ but not to $p_1$, then $p_2yz$ is a feasible $(1,3)$-2-chord with $|L(y)|=3$, contrary to Lemma \ref{lemma-13chord}. If $z$ is adjacent to $p_1$ but not $u$, then $W'=uxyz$ is a feasible 3-chord and $(G_{W',2},P_{W'},L)$ has a bad vertex $v$, where $P_{W''}=uxyzz'$. If $v$ is adjacent to $y$, then $v$ is a boundary vertex and $p_2yv$ is a $(1,3)$-2-chord with $|P \cap B(G_{W,2})|=2$ and $|L(y)|=3$, contrary to Lemma \ref{lemma-13chord}.   Hence $v$ is not adjacent to $y$. If $v$ is adjacent to $u$ and $z$, or $v$ is adjacent to $x$ and $z'$, then there  is a $(3,2)$-2-chord or  a $(2,2)$-3-chord, a contradiction to Lemma \ref{lemmanoichord}.  If $v$ is adjacent to $u$ and $z'$, then since $G$ has no $(2,2)$-2-chord, $v$ is a boundary vertex. Hence the interior of the cycle $uxyzz'v$ contains a unique vertex $v'$ adjacent to $v,x,z$. But then $z'zv'v$ is a bad 4-cycle. 
 	Thus $v$ is adjacent to $x$ and $z$. Then $W'''=uxvz$ is a feasible 3-chord, and $(G_{W''',2},P_{W'''},L)$ has a bad vertex $v'$. As $d_G(v) \ge 3$, $v'$ is adjacent to $v$ and is a boundary vertex. Then $v'vz$ is a feasible $(3,3)$-2-chord with $|L(v)|=3$, contrary to Lemma \ref{lemmano332chord}.   }
 	
 Thus $z$ is adjacent to both $u$ and $p_1$.

{Since $x, y$ are both interior vertices and adjacent and $|L(y)|=3$, we know that $|L(x)|=4$. Hence $d_G(x) \ge 4$. This implies that $G_{W,1}$ contains no vertex $y$ with $\{u, p_i\} \subseteq N_G(y)  $ for some $i \in \{3,4\}$.  So the argument above applies to $G_{W,1}$ as well. Therefore 
  $B(G_{W,1})=p_5p_4p_3xuz'$ is a 6-cycle, whose interior contains a unique vertex $y'$ which is adjacent to $p_4,x,z'$. 
}

 {
The structure of $G$   and the lists are completely determined. It is easy to find an orientation of $G-E(P)$ so that each vertex $v$ has out-degree   $|L(v)|-1$. As $G$ is bipartite, this implies that $G$ is $L$-colourable.
 }
  	\end{proof}
 Combined with Lemma \ref{lemmanoichord}, we have the following corollary.
 
\begin{corollary}
\label{cori22chord}
$G$ has no $(1^+,2)$-$2$-chord.
\end{corollary}

  \begin{lemma}
     \label{lemma112chord}
   There is no $(1,1)$-$2$-chord.
  	\end{lemma}	
  	\begin{proof}
  	{	Assume to the contrary that $W=p_ixp_j$ is a $2$-chord. Let $G_{W,2}=G-\{p_{i+1}, \ldots, p_{j-1}\}$ and $P_W=p_1\ldots p_ixp_j \ldots p_5$. We choose such a $2$-chord so that $G_{W,2}$ has minimum number of vertices.  Note that $|L(x)|=4$, for otherwise $x$ is a bad vertex.}

  If $|V(P)|=k  \leq 4$ or $j-i\geq 3$, then $W$ is a feasible chord with $|P_W| \le 4$, contraty to Lemma \ref{lemmapw5}.  Thus we assume that $P=p_1p_2p_3p_4p_5$ and $j-i =2$. By symmetry, we may assume that $i \le 2$, i.e., $W=p_1xp_3$ or $W=p_2xp_4$.

  If $W=p_1xp_3$,  we pre-colour  $x$ and let $G'= G-\{p_2\}$ and $P'=p_1xp_3p_4p_5$, and let $L'$ be the restriction of $L$ to $G'$, except that $x$ is pre-coloured.  As $|L'(v)| \ge 2$ implies that $|L(v)| = |L'(v)|$,  $(G', P', L')$  has no bad $4$-cycle and no bad edge.   Hence $(G',P',L')$   has a bad vertex $y$.
As $G$ has no separating $4$-cycle, we conclude that   $y$ has exactly two neighbours in $P'$. As $y$ is not a bad vertex in $(G,P,L)$, we know that  $N_G(y) \cap P' =\{x, p_4\}$ or $\{x, p_5\}$. 

If $y$ is a boundary vertex, then $p_1xy$ is a $(1,3)$-2-chord that contradicts Lemma \ref{lemmanoichord}. So $y$ is an interior vertex. 
 If $y$ is adjacent to $x$ and $p_5$, then  $W'=p_1xyp_5$  is a feasible chord with $|P_{W'}| =4$, contrary to Lemma \ref{lemmapw5}. If $y$ is adjacent to $x$ and $p_4$, then  $W''= p_1xyp_4$ is a feasible 3-chord with $P_{W''}=p_1xyp_4p_5$, and $(G_{W,2}, P_{W''}, L)$ has a bad vertex $z$. 
	As $d_G(y) \ge 3$, we know that $z$ is a boundary vertex adjacent to $y$. Then $W''' = zyp_4$ is a $(1,3)$-2-chord with $|P \cap B(G_{W''',2})| =2$ and $|L(y)|=3$, contrary to Lemma \ref{lemma-13chord}.

{ Therefore  $W=p_2xp_4$.  Then   $(G_{W,2}, P_W, L)$ has a bad  vertex $y$.
  If $y$ has two neighbours, say $p_i, p_j$,  in $P'-\{x\}$, then   $W'=p_iyp_j$ is a $(1,1)$-$2$-chord for which $G_{W',2}$ has fewer vertices than $G_{W,2}$, contrary to the choice of $W$. So  $y$ is adjacent to $x$  and
   has exactly one neighbour  in $P'-\{x\}$.  By symmetry, we may assume that $y$ is adjacent to $x$ and $p_1$. }
   
 {  If  $y$ is an interior vertex, then $W'=p_1yxp_4$ is a feasible 3-chord, and $(G_{W',2},P_{W'},L)$ has a bad vertex $z$. As $d_G(y) \ge 3$, $z$ is a boundary vertex adjacent to $y$.
   Then $W'''=p_1yz$ is a $(1,3)$-2-chord with $|P \cap B(G_{W''', 2})| =1$, that contradicts Lemma \ref{lemmanoichord}.
   Therefore  $y$ is a boundary vertex.}
   
  {Let $k \ge 0$ be the maximum integer such that $p_1y_0y_1\ldots y_{2k}$ is a subpath of $B(G)$, $x$ is adjacent to $y_0,y_2,\ldots, y_{2k}$, $|L(y_{2i})|=3$ and $|L(y_{2i-1})|=2$ for $i=1,2,\ldots, k$. Such a path exists, as $p_1y_0=p_1y$ is such a path. If $y_{2k}$ is adjacent to $p_5$, then $B(G)=y_0\ldots y_{2k}p_1\ldots p_5$ with a unique interior vertex $x$ adjacent to every other vertex of the boundary cycle. Orient the edges of $G-E(P)$ so that $p_i$ are sinks, and $y_0\ldots y_{2k}$ is a directed path, $x$ has out-neighbour $y_0$ and in-neighbours  $y_{2i}$ for $i=1,2,\ldots, k$. Then each vertex $v$ has out-degree $|L(v)|-1$. Hence $G$ is $L$-colourable. }
   
{Thus $y_{2k}$ is not adjacent to $p_5$. Then $W'=p_4xy_{2k}$ is a feasible $(1,3)$-2-chord.   Then $P_{W'}=p_5p_4xy_{2k}y_{2k}^+$ and 
   	$(G_{W',2}, P_{W'}, L)$ has a bad vertex $z$. } 

\medskip
{\bf Case 1}   $z$ is an interior vertex. 

     If $z$ is adjacent to $y_{2k}^+$ and $p_4$ or $p_5$, then $p_4zy_{2k}^+$ or $p_5zy_{2k}^+$ is a 2-chord that contradicts Corollary \ref{cori22chord}.  If $z$ is adjacent to $y_{2k}$ and  $p_5$, then $W' = p_5zy_{2k}$ is a 2-chord with $|P_{W'}| \le 4$,  a contradiction. If $z$ is adjacent to $p_4$ and $y_{2k}$, then $W'=p_4zy_{2k}$ is a feasible $(1,3)$-2-chord with $|L(z)|=3$, contrary to Lemma \ref{lemma-13chord}. 
     
     Thus $z$ is adjacent to $x$ and exactly one of $p_5$ and $y_{2k}^+$. 
 If $z$ is  
     adjacent to $x$ and   $p_5$, then $w''=p_5zxy_{2k}$ is a feasible 3-chord, with $P_{W''}= p_5zxy_{2k}y_{2k}^+$, and $(G_{W'',2}, P_{W''}, L)$ has a bad vertex $z'$. If $z'$ adjacent to $x$ and $y_{2k}^+$. Then $W'''=p_5zxz'y_{2k}^+$ is a feasible 4-chord, with $P_{W'''}=W'''$. Let $z''$ be a bad vertex of $(G_{W''',2},P_{W'''},L)$. As $d_G(z), d_G(z') \ge 3$, $z''$ is adjacent to at least one of $z,z'$. Thus $z''$ is a boundary vertex, and $p_5zz''$ or $y_{2k}^+z'z''$ is a 2-chord that contradicts Lemma \ref{lemmanoichord}.   Therefore $z'$ is adjacent to $z$, and hence $z'$ is a boundary vertex and $W'''=p_5zz'$ is a 2-chord that contradicts Lemma \ref{lemmanoichord}.
     
     If $z$ is adjacent to $x$ and $y_{2k}^+$, then  
     $w''=p_4xzy_{2k}$ is a feasible 3-chord, with $P_{W''}= p_5p_4xzy_{2k}^+$, and $(G_{W'',2}, P_{W''}, L)$ has a bad vertex $z'$. Similarly,  $z'$ is a boundary vertex adjacent to $z$ and $W'''=y_{2k}^+zz'$ is a 2-chord that contradicts Lemma \ref{lemmanoichord}.

\medskip
{\bf Case 2}   $z$ is a boundary vertex. 
 
     If $z$ is adjacent to $y_{2k}^+$ and $p_5$, then $p_5zy_{2k}^+y_{2k}xp_4$ is a separating 6-cycle, whose interior contains a unique vertex $z'$ which is adjacent to $z, p_4,y_{2k}$ (note that $d_G(z) \ge 3$).  But then $z'zy_{2k}^+y_{2k}$ is a bad 4-cycle, a contradiction. 
    If $z$ is adjacent to $p_4$ or $y_{2k}$, then $zp_4$ or $zy_{2k}$ is a chord that contradicts Lemma \ref{lemmanoichord}.  
     
     Therefore 
       $z$ is  adjacent to $x$ and one of $p_5$ and $y_{2k}^+$.   
By the maximality of $k$, $z$ is not adjacent to $y_{2k}^+$. Hence $z$ is adjacent to $p_5$. 
     
   	Let $k'$ be the largest integer such that $p_5z_0z_1\ldots z_{2k'}$ is a subpath of $B(G)$,    $x$ is adjacent to $z_0,z_2,\ldots, z_{2k'}$, $|L(z_{2i})|=3$ and $|L(z_{2i-1})|=2$ for $i=1,2,\ldots, k'$. Such a path exists, as $p_5z=p_5z_0$ is such a path. By the maximality of $k'$, $z_{2k'}$ is not adjacent to $y_{2k}^+$. So $W''=z_{2k'}xy_{2k}$ is a feasible 2-chord, and $(G_{W'',2},P_{W''}, L )$ has a bad vertex $v$, where $P_{W''}=y_{2k}^+y_{2k}xz_{2k'}z_{2k'}^-$. 
   
  {If $v$ is adjacent to   $y_{2k}^+$, then by using Corollary \ref{cori22chord}, we know that $v$ is not adjacent to $z_{2k'}$, and by using Lemma \ref{lemmasep6}, we can show that $v$ is not adjacent to $z_{2k'}^-$. Hence $v$ is adjacent to $x$.   Then $W'''=y_{2k}^+vxz_{2k'}$ is a feasible 3-chord, and $(G_{W''',2}, P_{W'''}, L)$ has a bad vertex $v'$, where $P_{W'''}= y_{2k}^+vxz_{2k'}w'$. Now $v'$ is adjacent to $v$ (for otherwise $d_G(v)=2$), and hence $v'$ is a boundary vertex. But then $v'vy_{2k}^+$ is a $(1^+,2)$-2-chord, in contrary to Corollary \ref{cori22chord}. So $v$ is not adjacent to $y_{2k}^+$. By the same argument, $v$ is not adjacent to $z_{2k'}^-$. Therefore $v$ is adjacent to $y_{2k}$ and $z_{2k'}$. Then $W''''=y_{2k}vz_{2k'}$ is a feasible $(3,3)$-2-chord with $|L(v)|=3$, contrary to Lemma \ref{lemmano332chord}.   }
 \end{proof}

\begin{lemma}
	\label{lemmanochord}
	$G$ has no chord.
\end{lemma}
\begin{proof}
	Assume $G$ has a chord.
	{	By  Lemma \ref{lemmanoichord} and Lemma \ref{lemmapw5}, we may assume that  $P=p_1p_2p_3p_4p_5$,  $W=p_3u$ and  $p_1,p_2,p_3 \in V(G_{W,2})$.  }
		
		  \begin{claim}
		  	\label{clm4}
		  {$|L(u)|=4$ and either $u$ is adjacent to $p_1$ or  $B(G_{W,2})$ is a 6-cycle $p_1p_2p_3uu^+x$ with a unique interior vertex $w$ adjacent to $u,p_2,x$.  }
		  \end{claim}
		  \begin{proof}
		{ 	  	Assume $u$ is not adjacent to $p_1$. Then $W$ is feasible,      $P_W=p_1p_2p_3uu^+$, and $(G_{W,2}, P_W, L)$ has a bad vertex $w$.}
	
	{    Then $|L(w)|=3$ and $w$ has one neighbour  in $\{p_1,p_2,p_3\}$ and one neighbour in $\{u, u^+\}$. If $w$ is adjacent to $u$ and $p_1$, then $W'=p_1wu$ is a feasible 2-chord with $|P_{W'}|=4$, contrary to Lemma \ref{lemmapw5}. If $w$ is adjacent to $u^+$ and $p_i$ for some $i \in \{1,2,3\}$, then $p_iwu^+$ is a $(1,2)$-2-chord, contrary to Lemma \ref{lemma-no(1,2)-2chord}. }

	 {    Thus $w$ is adjacent to $u$ and $p_2$.  Then $W'=p_2wu$ is a feasible 2-chord, with $P_{W'}=p_1p_2wuu^+$ and $(G_{W',2},P_{W'},L)$ has a bad vertex $x$. As $d_G(w) \ge 3$, $x$ is a boundary vertex adjacent to $w$ and at least one of $u^+$ and $p_1$. }

	    {  If $x$ is adjacent to $u^+$ but not to $p_1$, then $W''=p_2wx$ is a feasible $(1,3)$-3-chord with $|L(w)|=3$ and $| P \cap B(G_{W'',2})|=2$, contrary to Lemma \ref{lemma-13chord}. If $x$ is adjacent to $p_1$ but not to $u^+$, then $xwu$ is a $(3,4)$-2-chord with $|L(w)|=3$, contrary to Lemma \ref{lemma-34chord}. } 
	    	
	    	{Note that since $G$ has no bad vertex and no bad 4-cycle, $|L(u)|=4$.}
      This completes the proof of Claim \ref{clm4}.
	    \end{proof}

{	By Symmetry,  either $G_{W,1}$  consists of a 6-cycle $p_5p_4p_3uu^-x'$ with a unique interior vertex $w'$ adjacent to $p_4,u$ and $x'$, or $G_{W,1}$ is a 4-cycle $p_5p_4p_3u$. In any case, $G$ is a bipartite graph and it is easy to find an orientation $D$ of $G-E(P)$ with
	$d_D^+(v)=|L(v)|-1$ for each vertex $v$ of $G$. Hence $G$ is $L$-colourable, a contradiction.  }
\end{proof}

\begin{definition}
	\label{def-semifan}
	A {\em semi-fan} is a graph obtained from a path $v_1v_2\ldots v_{2q+1}$ (for some $q \ge 1$) by adding a vertex $u$ adjacent to $v_1, v_3, \ldots, v_{2q+1}$. The vertex $u$ is called the {\em center} of the semi-fan, and $v_1, v_{2q+1}$ are the {\em end vertices} of the semi-fan.
\end{definition}

\begin{lemma}
	\label{lemma-semifan}
	If $W=uvw$ is a $(3^+,3)$-$2$-chord of $G$, then $G_{W,2}$ is a semi-fan with center $v$.
\end{lemma}
\begin{proof}
	Assume the lemma is not true, and $W=uvw$ is a $(3^+,3)$-$2$-chord for which $G_{W,2}$ is not a semi-fan, and subject to this, $G_{W,2}$ has minimum number of vertices.
	
	If $W$ is not a feasible $2$-chord, then it must be the case that $u^+=w^-$, and hence $G_{W,2}$ is the $4$-cycle $uvww^-$, which is a semi-fan (with $q=1$) with center $v$.
	Assume $W$ is a feasible $2$-chord. Let $\phi$ be an $L$-colouring of $G'_{W,1}-E(P)$.
	
	By Lemma \ref{lemmapw5}, $|P_W|=5$. This implies that $|L(u^+)|=|L(w^-)|=2$ and $u^+ \ne w^-$.
	By Lemma \ref{lemmabadv}, $(G_{W,2},P_W, L_{\phi})$ has a bad vertex $x$. If $x$ is adjacent to $u$ and $w$, then $W'=uxw$ is a $(3^+,3)$-$2$-chord with $|L(x)|=3$, { contrary to Lemma \ref{lemmano332chord} or Lemma \ref{lemma-34chord}. }
	If $x$ is adjacent to $u$ and $w^-$, then $W'=uxw^-$ is a $(2,3^+)$-$2$-chord, contrary to Lemma \ref{lemmanoichord}.
		If $x$ is adjacent to $u^+$ and $w^-$, then since $G$ has no $(2,2)$-$2$-chord, $x$ is a boundary vertex of $G$. Then $C=xu^+uvww^-$ is a $6$-cycle and $d_G(x) \ge |L(x)| = 3$. Hence either $x$ is adjacent to $v$, and hence $G_{W,2}$ is a semi-fan, or $C$ has a unique interior vertex $y$ adjacent to $x,u,w$. This is in contrary to the previous paragraph for the case that $x$ is adjacent to $u$ and $w$.
	
 {	Therefore $x$ is either adjacent to both $w^-$ and $v$, or adjacent to both $u^+$ and $v$.}
	
 {	By symmetry, we assume that $x$ is adjacent to $w^-$ and $v$. If  $x$ is a boundary vertex of $G$, then
	$W'=uvx$ is a $(3^+,3)$-$2$-chord. By the minimality of $W$, $G_{W',2}$ is a semi-fan with center $v$.  But then $G_{W,2}$ is also a semi-fan with center $v$.}
	
 {	 Assume $x$ is an interior vertex. Let $W'=w^-xvu$. As $w^- \ne u^+$, $W'$ is a feasible $3$-chord with $P_{W'} = w^-xvuu^+$. Let $\psi$ be an $L$-colouring of $G_{W',1}-E(P)$.
	 By Lemma \ref{lemmabadv}, $(G_{W',2},P_{W'}, L_{\psi})$ has a bad vertex $x'$. If $x'$ is adjacent to $x$, then since $|L(x')|=|L(x)|=3$ and $x$ is an interior vertex, we conclude that $x'$ is a boundary vertex. But then $w^-xx'$ is a $(2,3)$-$2$-chord of $G$, contrary to Lemma \ref{lemmanoichord}. If $x'$ is adjacent to $u^+$ and $w^-$, then the $6$-cycle $uu^+x'w^-wv$ contains the unique interior vertex $x$, which is adjacent to $w^-$ and $v$. Hence $d_G(x)=2 < |L(x)|=3$, a contradiction.}
	
 {	Since $G$ has no separating $5$-cycle, $x'$ cannot be adjacent to $u$ and $w^-$. Therefore $x'$ is adjacent to $u^+$ and $v$. If $x'$ is a boundary vertex, then $W''=wvx'$ is a $(3,3)$-$2$-chord with $G_{W'',2}$ smaller than $G_{W,2}$. Hence $G_{W'',2}$ is a semi-fan, contrary to the fact that $x$ is an interior vertex.}
	
 { Assume $x'$ is an interior vertex. Let $W'''=u^+x'vxw^-$. Then $P_{W'''}=W'''$.  Let $\tau$ be an $L$-colouring of $G_{W''',1}-E(P)$.
	 By Lemma \ref{lemmabadv}, $(G_{W''',2},P_{W''}, L_{\tau})$ has a bad vertex $x''$. If $x''$ is adjacent to $u^+$ and $w^-$, then the $6$-cycle $x''w^-wvuu^+$ contains two interior vertices, contrary to Lemma \ref{lemmasep6}. If $x''$ is adjacent to $x$, then since $|L(x'')|=|L(x)| =3$, we conclude that $x''$ is a boundary vertex. Hence $x''xw^-$ is a $(2,3)$-$2$-chord, contrary to Lemma \ref{lemmanoichord}. So $x''$ is not adjacent to $x$. Similarly $x''$ is not adjacent to $x'$. Therefore $x''$ is adjacent to $v$ and $w^-$, or adjacent to $v$ and $u^+$. In any case, $G$ has a separating $4$-cycle, contrary to Lemma \ref{lemmasep45}.}
	\end{proof}

\section{Proof of Theorem}

 Assume $B(G)=(p_kp_{k-1}\ldots p_1 v_1v_2\ldots v_s)$.

\begin{claim}
	\label{clm-nobadedge}
	For $i \ge 3$,  $v_1$ and $p_i$ have no common inner  neighbour, and for $i \le k-2$, $v_s$ and $p_i$ have no common inner neighbour.
\end{claim}
\begin{proof}
	If  $x$ is adjacent to $v_1, p_k$, then $W=v_1xp_k$ is a $(1,3^+)$-$2$-chord with $|P \cap B(G_{W,2})|=1$, contrary to Lemma \ref{lemmanoichord}.

	If   $x$ is adjacent to $v_1, p_4$, then $C=v_1p_1p_2p_3p_4x$ is a $6$-cycle. By
 Lemma \ref{lemmasep6},  ${\rm int}(C)$ has at most one vertex. If ${\rm int}(C)$ has one vertex $z$,
 then $z$ has three neighbours in $C$, and hence is adjacent to two pre-coloured vertices. Then   $G$ has a $(1,1)$-$2$-chord, contrary to Lemma
 \ref{lemma112chord}. Thus   ${\rm int}(C)$ has no vertex. We add an edge $e=p_2p_4$ and let $G'= G-\{p_3\}$.   Let $P'=p_1p_2p_4p_5$
 and $L'(v)=L(v)$ for any vertex $v \in V(G')$.  Hence $(G', P', L)$  is a valid target. By the minimality of $G$,
   $G'-E(P')$ is $L'$-colourable. Therefore  $G-E(P)$ is $L$-colourable,  a contradiction.
	
If   $x$ is adjacent to $v_1, p_3$, then $C=v_1p_1p_2p_3x$ is a $5$-cycle. As $G$ has no separating $5$-cycle,  ${\rm int}(C)$ has no vertex. We add an edge $e=p_1p_3$ and let $G'= G-\{p_2\}$.   Let $P'=p_1p_3p_4 p_5$ and $L'(v)=L(v)$ for any vertex $v \in V(G')$.  Hence $(G', P', L)$  is a valid target. By the minimality of $G$,
  $G'-E(P')$ is $L'$-colourable. Therefore  $G-E(P)$ is $L$-colourable,  a contradiction.
\end{proof}

\begin{claim}
	\label{clm2}
Assume $X $ is a set of consecutive vertices on $ B(G)$,   $G'=G-X$, $\phi$ is a colouring of $X$, and    $L'$ is obtained from the restriction
of $L$ to $G'$ by removing, for each $v \in X$,   the colour of $v$ from $L(u)$ for each neighbour $u$ of $v$.
If $C=y_1y_2y_3y_4$  is a bad $4$-cycle in $(G',P, L')$ and  only one vertex $v \in   X$ has a neighbour $y_i$ in $C$, then $vy_i$ is a boundary edge of $G$, and   $|L'(y_i)| = |L(y_i)|-1=3$.
\end{claim}
\begin{proof}
	Since $G$ has no bad $4$-cycle, some vertex in $X$ has a neighbour in $C$. By our assumption this vertex $v$ is unique.
	
	Assume $y_1,y_2,y_3$ are three consecutive vertices on $B(G')$ and $y_4$ is an interior vertex of $G'$ (hence $|L'(y_2)|=2$ and $|L'(y_1)| = |L'(y_3)| = |L'(y_4)|=3$).
	
   Since $y_4$ is an interior vertex of $G'$, $y_4$ is not adjacent to $v$. Hence $|L(y_4)| = |L'(y_4)| =3$.
   If $y_1$ is an interior vertex of $G$, then $|L(y_1)| =4$ (as $X_{G,L}$ is an independent set). Hence $L'(y_1) \ne L(y_1)$,
   and $e=vy_1$ is an edge of $G$. Then $v$ is not adjacent to $y_2$ (as $G$ is triangle free).
   Hence $y_2$ is a boundary vertex with $|L(y_2)|=|L'(y_2)|=2$. But then $vy_1y_2$ is a $(1^+,2)$-$2$-chord, contrary to Corollary \ref{cori22chord}.
   Therefore $y_1$ is a boundary vertex. Similarly, $y_3$ is a boundary vertex. If $y_2$ is an interior vertex, then  $W_1=y_1y_2y_3$ and $W_2=y_1y_4y_3$ are  $(3,3)$-$2$-chords. By Lemma \ref{lemma-semifan}, $G_{W_1,2}$ and $G_{W_2,2}$ are semi-fans, and hence contain no interior vertex. However, $y_4$ is an interior vertex of $G_{W_1,2}$ or $y_2$ is an interior vertex of $G_{W_2,2}$, a contradiction. So $y_2$ is a boundary vertex of $G$. As $G$ has no chord, $y_1y_2,y_2y_3$ are boundary edges of $G$.

   This implies that $vy_2$ is not an edge. By symmetry, we may assume that $vy_1$ is an edge of $G$. As $G$ has no chord, $vy_1$ is a boundary edge, with $|L'(y_1)| =|L(y_1)|-1 =3$.
\end{proof}

\begin{claim}
	\label{clm-badedge}
	Assume $X $ is a set of at most three consecutive vertices on $ B(G)$ satisfying one of the following conditions:
		\begin{itemize}
			\item[(1)] $X=\{v_1\}$.
			\item[(2)] $X=\{v_1, v_2\}$, $L(v_1)=L(v_2)\cup L(p_1)$ and $|L(v_2)|=2$.
			\item[(3)] $X=\{v_1, v_2, v_3\}$, $L(v_1)=L(v_2)\cup L(p_1)$ with $|L(v_2)|=2$,  $L(v_2) \subseteq L(v_3)$ and  $|L(v_4)|=2$.
		\end{itemize}
	Assume $G'=G-X$, $\phi$ is a colouring of $X$, and    $L'$ is obtained from the restriction
	of $L$ to $G'$ by removing, for each $v \in X$,   the colour of $v$ from $L(u)$ for each neighbour $u$ of $v$.
		Then if $e=xy$ is a bad of $(G', P, L')$, then $e$ is a boundary edge of $G$;    if $x$  is a bad vertex in $(G', P, L')$, then $x$ is adjacent to a vertex of $X$ and a vertex in $P$.
\end{claim}
\begin{proof}
If $x$ is a bad vertex in $(G', P, L')$, then as $(G,P,L)$ has no bad vertex, it follows that $x$ is adjacent to a vertex in $X$ and a vertex in $P$. It remains to show that $(G', P, L')$ has no bad edge.

 Assumme $e=xy$ is a bad edge and $e$ is not a boundary edge of $G$. If $X=\{v_1\}$, then since $(G,P,L)$ has no bad edge and $G$ is triangle free, exactly one of $x,y$ is adjacent to $v_1$, say $xv_1 \in E(G)$. Then $L(y)=L'(y)$ and $|L(y)|=2$.  Hence $y$ is a vertex on $B(G)$. Since $xy$ is not a boundary edge of $G$, $x$ is an interior vertex and  $v_1xy$ is $(2,3^+)$-$2$-chord, contratry  to Corollary \ref{cori22chord}.  

Assume $X=\{v_1,v_2\}$.   If both $x$ and $y$ are adjacent to vertices  of $X$, then we assume $x$ is adjacent to $v_1$ and $y$ is adjacent to $v_2$.   As $X_{G,L}$ is an independent set,  $y$ is a vertex on $B(G)$.   As $G$ has no chord,   $ y=v_3$ and  $|L(v_3)|=3$. In this case,  $v_1xy_3v_4$ is a bad $4$-cycle in $(G,P,L)$, a contradiction.     Hence  only one of $x$ and $y$ is adjacent to a vertex of $X$. We assume $x$ is adjacent to  a vertex of $X$  and $L(y)=L'(y)$.  Then $|L(y)|=2$ and $y$ is a vertex on $B(G)$. Similarly, $x$ is an interior vertex and  $v_1xy$ is $(2,3^+)$-$2$-chord, contrary  to Corollary \ref{cori22chord}.

  Assume $X=\{v_1,v_2,v_3\}$,  $L(v_1)=L(v_2)\cup L(p_1)$ with $|L(v_2)|=2$,  $L(v_2) \subseteq L(v_3)$ and  $|L(v_4)|=2$.  If both $x$ and $y$ are adjacent to vertices  of $X$, as $X_{G,L}$ is an independent set,  then one of $x$ and $y$ is a boundary vertex on $B(G)$.  We assume $y$ is  the boundary vertex on $B(G)$.   As $G$ has no chord,   $ y=v_4$ and  $v_1xyv_3v_2$ is $5$-cycle or $v_2xyv_3$ is a $4$-cycle. In the any case $d_G(v_3)=2 < |L(v_3)|$, contrary to Lemma \ref{lemmaldv}.    Hence  only one of $x$ and $y$ is adjacent to a vertex of $X$. We assume $x$ is adjacent to  a vertex of $X$  and $L(y)=L'(y)$.  Then $|L(y)|=2$, and $y$ is a vertex on $B(G)$. Similarly, $x$ is an interior vertex and  $v_1xy$ is $(2,2^+)$-$2$-chord, contrary  to Corollary \ref{cori22chord}.  
\end{proof}

\bigskip
\noindent
{\bf Case 1} $s =0$ or $s\ge 2$  and $|L(v_2)| \ge 3$.

 	Let $L'$ be the list assignment of $G'=G- \{p_1\}$ obtained from $L$ by deleting the colour of $p_1$ from the lists of   neighbors of $p_{1}$.
 By Lemmas \ref{lemma-no(1,2)-2chord} and \ref{lemma112chord},
 $(G', P-\{p_1\}, L')$ has no bad edge and no bad vertex.
 By  Claim \ref{clm2}, if $(G', P-\{p_1\}, L')$ has a bad $4$-cycle, then   $s \ge 3$ and $|L(v_2)|=2$, contrary to our assumption.
 	
 Hence $(G', P-\{p_1\}, L')$ is  a valid target, and  $G'-E(P-\{p_1\})$ has an $L'$ colouring, which extends to an $L$-colouring of $G-E(P)$. This completes the proof of Case 1.

 	\bigskip
 	\noindent
 	{\bf Case 2} $s =1$ or    $s \ge 2$ and $|L(v_2)|=2$ and $L(v_1) - (L(v_2) \cup L(p_1)) \ne \emptyset$.
 		
   We colour $v_1$  by a colour  $c \in  L(v_1)-(L(p_1) \cup L(p_k))$ (if $s=1$)   or $c  \in L(v_1)-(L(v_2) \cup L(p_1))$ (if $s \ge 2$) and
   delete the colour of $v_1$ from the lists of neighbours of $v_1$.
   Let $L'$ be the resulting list assignment of $G'=G-\{v_1\}$.
 	It is obvious that  $(G',P,L')$ has no bad 4-cycle and no bad edge.

 If $(G',P,L')$ has a bad vertex $x$, then  $x$ is adjacent to $v_1$.  By Lemma \ref{lemma112chord},   $(G,P,L)$ has no $(1,1)$-$2$-chord.  So $x$ has one neighbour in $P$, with $|L'(x)|=2$ (and hence $|L(x)|=3$).  By Claim \ref{clm-nobadedge}, $x$ is adjacent to $p_2$.   Let  $G''=G'-\{ p_1\}$.   We pre-colour $x$ and let $L''$ be the resulting list assignment of $G''$.  Let $P''$ be the pre-coloured  vertices in $G''$, i.e.,
$P''=xp_2, \ldots, p_5$.    If $(G'',P'',L'')$ has a bad vertex $x'$,  then $x'$ is adjacent to $x$, hence $|L(x')|=4$. This implies that $x'$ has two neighbours in $P$,   contrary to Lemma \ref{lemma112chord}. Hence $(G'',P'',L'')$ has no bad vertex.   Again it is easy to see that   $(G',P'',L'')$ has no bad $4$-cycle and no bad edge.  Therefore $G''-E(P'')$ has an $L''$-colouring $\psi$,   which  extends to an $L$-colouring of $G-E(P)$.
This completes the proof of Case 2.

\bigskip

In the remainder of this section, we assume that $s \ge 2$, $|L(v_2)|=2$ and $L(v_1)=L(v_2) \cup L(p_1)$.

 \bigskip
 \noindent
 {\bf Case 3}    $L(v_2) \not\subseteq L(v_3)$ or   $|L(v_4)| \ge 3$.
 		
 	If  $L(v_2) \not\subseteq L(v_3)$, then we colour $v_2$ by a colour $c_2 \in L(v_2) - L(v_3)$ and colour $v_1$ by a colour $c_1 \in L(v_1) - (L(p_1) \cup \{c_2\})$, and for $i=1,2$, delete colour $c_i$ from the lists of neighbours of $v_i$.
 	If   $|L(v_4)| \ge 3$, then we colour $v_2$ by a colour $c_2 \in L(v_2)  $ and colour $v_1$ by a colour $c_1 \in L(v_1) - (L(p_1) \cup \{c_2\})$,  and for $i=1,2$, delete colour $c_i$ from the lists of neighbours of $v_i$.

 Let  $G'=G-\{v_1, v_2\}$ and $L'$ be the resulting list assignment of $G'$.   It is obvious that $(G',P,L')$ is a target, and has no bad edge.

 Assume $(G', P, L')$ has a bad $4$-cycle $C=y_1y_2y_3y_4$  with $|L'(y_1)|=|L'(y_3)|=|L'(y_4)|=3$ and $|L'(y_2)|=2$ and  $y_4$ be an interior vertex of $G'$  and $y_1,y_2,y_3$ be consecutive vertices on  $B(G')$.

 Since $G$ has no bad $4$-cycle,    there is at least one edge between $\{v_1, v_2\}$ and   $\{y_1, y_2, y_3\}$.  As $G$ has no triangle and no separating $4$-cycle,  each of $v_1$ or $v_2$ is adjacent to at most one of the $y_i$'s.

 \bigskip
 \noindent
 {\bf Case 3(i)}  $v_1$ is adjacent  $y_i$ and   $v_2$ is adjacent  $y_j$.

 Since $G$ has no separating $5$-cycle,  $y_iy_j \in E(G)$. So by symmetry, we may assume that $v_1y_1, v_2y_2 \in E(G)$.
 Then $y_3 \in B(G)$. If $y_2$ is an interior vertex of $G$, then
  $W=v_2y_2y_3$ is a $(2,3^+)$-$2$-chord, contrary to Corollary \ref{cori22chord}. Thus we assume that $y_2$ is a boundary vertex. As $G$ has no chord, the edge $v_2y_2$ is a boundary edge, i.e., $y_2=v_3$ and $y_3=v_4$. We must have $L(v_2) \subseteq L(v_3)$, for otherwise, we could have chosen the colour $c \in L(v_2)-L(v_3)$ and hence $|L'(v_3)|=|L(v_3)| \ge 3$, contrary to the assumption that $|L'(y_2)|=2$. Let $G''$ be obtained from $G$ by identifying $v_1$ and $v_3$ into a single vertex $v^*$. Since $v_1, v_3$ have a common neighbour $y_1$, we know that $G''$ is triangle free (for otherwise $G$ would have a separating $5$-cycle). Let $L''(v)=L(v)$ for all $v \ne v^*$ and $L''(v^*) = L(v_2) \cup L(p_1)$. Then $(G'', P, L'')$ is a valid target, and hence $G''$ has an $L''$-colouring $\psi$. Since $\psi(p_1) \in L(p_1)$, we know that $\psi(v^*) \in L(v_2) \subseteq L(v_1) \cap L(v_3)$. Thus $\psi$ can be extended to an $L$-colouring of $G$ by assigning colour $\psi(v^*)$ to $v_1$ and $v_3$.

  \bigskip
  \noindent
  {\bf Case 3(ii)} Exactly one of $v_1$ and $v_2$ is adjacent to one of the $y_i$'s.

If $y_1$ is an interior vertex of $G$, then $|L(y_1)| = 4$ (as $X_{G,L}$ is an independent set). Hence $y_1$ is adjacent to $v_1$ or $v_2$, and hence $v_1y_1y_2$ or $v_2y_1y_2$ is a $2$-chord, contrary to Corollary \ref{cori22chord} (note that $L(y_2)=L'(y_2)$).

Thus $y_1$ is a boundary vertex of $G$, and similarly $y_3$ is  boundary vertex of $G$. If $y_2$ is an interior vertex, then $y_2$ is adjacent to $v_i$ for some $i \in \{1,2\}$. Then $W_1=v_iy_2y_1$ and $W_2=v_iy_2y_3$ are $2$-chords of $G$. By Corollary \ref{cori22chord},
both are $(3^+,3^+)$-$2$-chords, and hence $v_i=v_1$.
By Lemma \ref{lemma-semifan}, $G_{W_1,2}$ and $G_{W_2,2}$ are semi-fans, and hence contains no interior vertex. However, $y_4$ is an interior vertex of $G_{W_1,2}$ or $G_{W_2,2}$, a contradiction.

So $y_1,y_2, y_3$ are consecutive  boundary vertices of $G$. As $G$ has no chord,    $y_2$ is not adjacent to $v_1$ or $v_2$, and $v_1$ is not adjacent to $y_1$ or $y_3$. By symmetry, we may assume that   $y_1v_2$ is a boundary edge and hence $y_1=v_3$ and $y_2=v_4$. There exists $c \in L(v_3)$ such that  $L'(v_3) = L(v_3)-\{c\}$, for otherwise, $y_1y_2y_3y_4$ is a bad $4$-cycle in $(G,P,L)$. This implies that $L(v_2) \subseteq L(v_3)$ (for otherwise we would have chosen $c \in L(v_2) -L(v_3)$). Hence by our assumption, $|L(v_4)| \ge 3$, contrary to the fact that $|L'(y_2)|=|L(y_2)|=2$.

Hence $(G',P,L')$ has no bad $4$-cycle. Therefore 
  $(G',P,L')$ has a bad vertex $x$. Then $x$ is adjacent to $v_1$ and $p_2$, and $|L(x)|=3$, $|L'(x)|=2$.   Let  $G''=G'-\{ p_1\}$, and we pre-colour $x$ and let $L''$ be the resulting list assignment of $G''$.  Let $P''$ be the pre-coloured  vertices in $G''$, i.e.,
$P''=xp_2 \ldots p_5$.    Assume $(G'',P'',L'')$ has a bad vertex $x'$.
Then   $x'$  is adjacent to $x$ and one vertex of $P$. Since $G$ has no chord, $x'$ is an interior   vertex. Hence $|L (x')| = 4$, a contradiction. Thus  $(G'',P'',L'')$ has no bad vertex.  

  As $L''(v)=L'(v)$ except that $x$ is a pre-coloured vertex and $(G', P, L')$ has no  bad $4$-cycle,   $(G'', P'',L'')$ has no bad $4$-cycle, no bad edge. Therefore $G''-E(P'')$ has an $L''$-colouring $\psi$,  which  extends to an $L$-colouring of $G-E(P)$, a contradiction.
This completes the proof of Case 3.

 		\bigskip
 		\noindent
 		{\bf Case 4}   $L(v_2) \subseteq L(v_3)$ and   $|L(v_4)| = 2$.

Note that if $y$ is a common neighbour of $v_3$ and $p_i$, then $W=p_iyv_3$ is feasible. So it follows from Lemma \ref{lemma-13chord} that for $i=1,2,4,5$, $v_3$ and $p_i$ have no common neighbour $y$ with $|L(y)|=3$.

 			 {If $L(v_2)=L(v_4)$, then colour $v_3$ by a colour $c \in L(v_3)-L(v_2)$.   Let $G'=G-\{v_3\}$   and let $L'$ be the restriction of $L$ to $G'$ by deleting $c$ from the lists of the neighbors of $v_3$. As $L'(v_2)=L(v_2)=L(v_4)=L'(v_4)$, $(G', P,L')$ is a target. It is obvious that  $(G', P,L')$ has no bad edge and no bad $4$-cycle.  If $(G', P,L')$ has no bad vertex, then 
      $G'-E(P)$ has an $L'$-colouring $\phi$, which extends to an $L$-colouring of $G-E(P)$. 
     Assume $(G', P,L')$ has a bad vertex $y$. As observed above,   $y$ is a common neighbor of $p_3$ and $v_3$ and $|L(y)|=3$.  
    In this case, we pre-colour $y$. The pre-coloured vertices divide $G'$ into two parts $G'_1$
    with pre-coloured vertices $P'_1=p_1p_2p_3y$ and $G'_2$ with pre-coloured vertices $P'_2=yp_3p_4p_5$. Let $L'_j$ be the restriction of $L$ to $G'_j$, except that $y$ is pre-coloured. Then it is obvious that $(G'_j,P'_j, L'_j)$  is a valid target, and hence $G'_j-E(P_j)$ has a proper $L'_j$-colouring, whose union is a proper $L'$-colouring of $G'-E(P)$, a contradiction. 
     
   { If $L(v_4) \ne L(v_2)$, then colour $v_1$ and $v_3$ by a same colour $c \in L(v_2) - L(v_4)$, and colour $v_2$ by the other colour in $L(v_2)$. Let $G'=G-\{v_1,v_2,v_3\}$ and let $L'$ be obtained from the restriction of $L$ to $G'$ by deleting the colours of $v_1, v_2, v_3$ from the lists of their neighbours.} 
 			
 			Then  $X_{G',L'} \subseteq X_{G,L}$ is an independent set.
 			
 		 Note that $v_1$ and $v_3$ may have a common neighbour $w$, but since $v_1$ and $v_3$ are coloured the same colour, we have $|L'(w)| \ge |L(w)|-1$. 	Thus  any vertex of $G'$ adjacent to a coloured vertex  is a boundary vertex of $G'$, and   has lost at most one colour from its list, Hence   $(G', P, L')$ is a target.  Similarly, $(G',P,L')$ has no bad edge.

\begin{lemma}
	\label{clm-no4cycle}
	$(G', P, L')$ has no bad $4$-cycle.
\end{lemma}

The proof of this lemma is left to the next section.
Assume Lemma \ref{clm-no4cycle} is true, we continue with the proof of
Case 4.

If $(G', P, L')$ has no bad vertex, then $(G', P, L')$ is a valid target and hence there is an $L'$-colouring of $G'-E(P)$, which extends to an $L$-colouring of $G-E(P)$.

Assume  $(G', P, L')$ has bad vertices $y_1,\ldots, y_q$ ($q \le 1$).     Then $|L(y_j)|=3$ and     $y_j$ is adjacent to one vertex in $\{v_1,v_2,v_3\}$ and one vertex in $p_j \in P$.  

If $y_j$ is adjacent to $v_2$ and $p_i$, then $W=p_jy_jv_2$ is a $(1,2)$-$2$-chord, contrary to Corollary \ref{cori22chord}.   By Claim \ref{clm-nobadedge}, if $y_j$ is adjacent to $v_1$, then $p_j=p_2$ and if $y_j$ is adjacent to $v_3$, then $p_j = p_3$. So $j \le 2$. For simplicity, we assume that $(G', P, L')$ has two bad vertices $y_1,y_2$.
 
We pre-colour $y_1, y_2$. Together with vertices in $P$, the pre-coloured vertices divide $G''$ into three parts: $G''_1,  G''_2, G''_3$. Let $P''_j$ be the pre-coloured vertices on the boundary of $G''_j$, and let $L''_j$ be the restriction of $L'$ to $G''_j$, except that vertices of $P''_j$ are pre-coloured. Then $(G''_j, P''_j, L''_j)$ is a   target for $j=1,2,3$. As $(G',P,L')$ has no bad $4$-cycle, each $(G''_j, P''_j, L''_j)$ has no bad $4$-cycle. By using the fact that $X_{G,L}$ is an independent set, it is easy to see that $(G''_j, P''_j, L''_j)$ has no bad vertex, and  no bad edge. Hence each $G''_j - E(P''_j)$ has an $L''_j$-colouring. The union of these colourings form an $L'$-colouring of $G'$, which extends to an $L$-colouring of $G$.


 \section{Proof of Lemma \ref{clm-no4cycle}}

  	Assume $(G', P, L')$ has a bad $4$-cycle $C=y_1y_2y_3y_4$  such that
  	\begin{itemize}
  		\item  $|L'(y_1)|=|L'(y_3)|=|L'(y_4)|=3$ and $|L'(y_2)|=2$.
  		\item  $y_4$ is  an interior vertex of $G'$.
  		\item $y_1,y_2,y_3$ are consecutive on  $B(G')$.
  	\end{itemize}
  	
  	Let $E' = E[\{v_1,v_2,v_3\}, \{y_1, y_2, y_3\}]$ be the set of edges between $\{v_1,v_2,v_3\}$ and $ \{y_1, y_2, y_3\}$.
  	
  	 Since $G$ has no bad $4$-cycle,  $|E'| \ne \emptyset$.  Since $G$ has no separating $4$-cycle or $5$-cycle, each $v_i$ is adjacent to at most one of the $y_j$'s. Hence $|E'| \le 3$.
  	
 If $|E'|=1$, then it  follows from Claim \ref{clm2} that $E'=\{v_3y_1\}$ and $v_3y_1$ is a boundary edge of $G$, i.e., $y_1=v_4$ and $|L(v_4)| =4$, contrary to our assumption.

 If $E'=\{v_1y_i, v_2y_j\}$, then since $G$ has no separating $5$-cycle,   $y_iy_j \in E(G)$.
  	So either $E'=\{v_1y_1, v_2y_2\}$ or $E'=\{v_1y_2, v_2y_3\}$.
  	In the former case, $y_3$ is a boundary vertex of $G$ (as $X_{G,L}$ is an independent set and $y_3y_4$ is an edge). Hence $v_2y_2y_3$ is a $(2,3)$-$2$-chord, a contradiction.
  		In the latter case, $y_1$ is a boundary vertex of $G$. Hence $v_1 y_2y_1$ is a $(3,3)$-$2$-chord. By Lemma \ref{lemma-semifan}, $G_{W,2}$ is a semi-fan. However, $y_4$ is an interior vertex  of $G_{W,2}$, a contradiction.
  	
   The case $E'=\{v_3y_i, v_2y_j\}$ is symmetric to the case that $E'=\{v_1y_i, v_2y_j\}$.
  	
  Assume $E'=\{v_1y_i, v_3y_j\}$.

    If $y_iy_j $ is not an edge, then $i=1, j=3$.
    Now $C=v_1v_2v_3y_3y_4y_1$ is a separating $6$-cycle. Hence $y_2$ is the unique interior vertex of $C$ and $y_2v_2$ is an edge, contrary to the assumption that $|E'|=2$.

    Assume $y_iy_j \in E(G)$, say $E'=\{v_1y_1, v_3y_2\}$.

    Then $W=v_3y_2y_3$  is a $(3^+,3)$-$2$-chord.
    By  Lemma \ref{lemma-semifan}, $G_{W,2}$ is a semi-fan with center $y_2$, and with $G_{W,2} - \{y_2\}$ be a path $x_1x_2 \ldots x_{2q+1}$ with $x_1=v_3, x_{2q+1}=y_3$.

     Since $v_1v_2v_3y_2y_1$ is a face, we conclude that $d_G(v_3) = 3$. Hence $|L(v_3)|=3$, and $v_3v_4v_4y_2$ is a bad $4$-cycle in $G$, a contradiction.

    In the remainder of this section, we assume that $|E'|=3$. Hence
     $$E'=\{v_1y_1, v_2y_2,v_3y_3\}.$$

   Note that  $C$ is the unique bad $4$-cycle of $(G',P,L')$. For otherwise, let  $C'=z_1z_2z_3z_4$ be another   bad $4$-cycle of $(G',P,L')$. Then each of $v_1, v_2, v_3$ must have a neighbour in  $C'$ as well as a neighbour in $C$. This is impossible, as $v_1, v_2, v_3$ are consecutive boundary vertices of $G$.

    We properly colour $y_2$ with colour $\alpha$, let $G''=G'-\{ y_2\}$ and let $L''$ be the restriction of $L'$ to $G''$ except that the colour $\alpha$  is removed from the lists of the neighbours of $y_2$. Note that $|L(y_1)|=|L(y_3)|=4$, and $|L''(y_1)|, |L''(y_3)| \ge 2$.

    We shall prove that $G''$ has an $L''$-colouring which extends to an $L$-colouring of $G$.

   It is easy to show that  $(G'', P, L'')$ has no bad $4$-cycle and no bad edge. However, $(G'', P, L'')$ may have  bad vertices.

  If $y$ is a bad vertex of $(G'', P, L'')$, then by Claim \ref{clm-badedge},   $y$ is adjacent to $v_1$ or $v_3$. If $y$ is adjacent to $v_1$, then $y$ is adjacent to $p_2$. By Lemma \ref{lemma-13chord}, if $y$ is adjacent to $v_3$, then $y$ is adjacent to $p_3$, unless $y=y_3$, in which case, $y$ might be adjacent to $p_4$. So $(G'', P, L'')$ has at most two  bad vertices. 
  
 Assume   $(G'', P, L'')$ has two bad vertices $y,y'$, with $y$ adjacent to $v_1$ and $p_2$,   $y'$ adjacent to $v_3$ and $p_3$.

  We  pre-colour $y,y'$, and divide  $G''$ into three  parts  $G''_1, G''_2,  G''_3$, with
   $P''_j=p_1p_2y, P''_2= yp_2p_3y'$ and $P''_3=y'p_3p_4p_5$.   It is easy to see that each $(G''_j, P''_j, L''_j)$ has no bad $4$-cycle and no bad edge. Also $(G''_1, P''_1, L''_1)$ has no bad vertex.

  If $z$ is a bad vertex in $(G''_2, P''_2, L''_2)$, then $|L(z)|=3$ and $z$ is adjacent to $y$ and $p_3$, or $z$ is adjacent to $y'$ and $p_2$.
  As $X_{G,L}$ is an independent set, in the former case, $|L(y)|=4$ and hence $y=y_1$, and in the latter case, $|L(y')|=4$ and hence $y'=y_3$. Note that we cannot have both $y=y_1$ and $y'=y_3$, for otherwise, $y_1y_2y_3p_3p_2$ is a separating 5-cycle, a contradiction. 
  
Assume $z$ is adjacent to $y=y_1$ and $p_3$. We pre-colour $z$ and replace $G''_2$ with $G''_2$ and with $P''_2=yzp_3y'$. Now   $(G'''_2, P'''_2, L'''_2)$ is a valid target. Since $y' \ne y_3$, $(G''_3, P''_3, L''_3)$ has no bad vertex and hence is a valid target.  So $G''_j-E(P''_j)$ has a proper $L''_j$-colouring for $j=1,2,3$, and their union is a proper $L$-colouring of $G-E(P)$.  The other cases are proved similarly, and the details are omitted.

\section*{Acknowledgement}

This is a revision on paper [J. Huand X. Zhu, 
List coloring triangle-free planar graphs. J. Graph Theory94(2020), no.2, 278–298.] 
We thank  Yueyue Zhao and Yiting Jiang from Nanjing Normal University for pointing out that our original proofs of Lemmas 8 and 9 (corresponding to Lemmas 11 and 12 in this revised version) omitted the case that the bad vertex $y$ might be a boundary vertex. To handle this case, we added some new lemmas. We changed the definitions of bad vertex and bad edge to improve the readability of the paper. Also there is an error in the proof of Case 4 of Theorem 5: the bad 4-cycle could be $v^*v_4v_5w'$ for some $w' \ne w$. In this revised version, the proof is re-written.

 \bibliographystyle{unsrt}

\begin{thebibliography}{99}

 \bibitem{Alon1992}N. Alon, M. Tarsi.  {\em Colorings and orientations of graphs}.  Combinatorica, 1992, 12(2): 125-134.

 \bibitem{DLM2017} Z. Dvo\v{r}\v{a}k, B. Lidicky and B. Mohar.     {\em 5-choosability of graphs with crossings far apart}. Journal of Combinatorial Theory, 2017, 123:54-96.

 \bibitem{ERT1979}P. Erd\H{o}s, A. L. Rubin  and H. Taylor,    {\em Choosability in graphs}. Congressus Numerantium, 26 (1980), pp. 125-157.


\bibitem{KT1994} J. Kratochiv\'{i}l, Z. Tuza.   {\em Algorithmic complexity of list colorings}. Discrete Applied Mathematics, 1994, 50(3): 297-302.



\bibitem{Gro1958}  H. Gr\"{o}tzsch. Zur Theorie der diskreten Gebilde. VII. Ein Dreifarbensatz f\"{u}r dreikreisfreie Netze auf der Kugel. Wiss. Z. Martin-Luther-Univ. Halle-Wittenberg. Math.-Nat. Reihe, 1958:109-120.



 \bibitem{Tho2000} C. Thomassen, {\em A short list color proof of Gr\"{o}tzsch's theorem}.  JCTB, 2003, 88(1): 189-192.
 	

 \bibitem{Tho1994} C. Thomassen,  {\em   Every Planar Graph Is 5-Choosable}. Journal of Combinatorial Theory, 1994, 62(1): 180-181.
 	
 	
 \bibitem{Viz1976}V. G. Vizing, {\em Vertex colorings with given colors (in Russian)}. Metody Diskretnogo Analiza, Novosibirsk, 29 (1976), pp. 3-10.
 	

 	
 \bibitem{Voi1995} M. Voigt.  {\em A not 3-choosable planar graph without 3-cycles}. Discrete Mathematics, 1995, 146(1-3):325-328.
 \end{thebibliography}

\end{document}